\documentclass[12pt,reqno]{amsart}

\newcommand{\F}{{\mathbb F}}

\newcommand{\cM}{{\mathcal M}}

\newcommand{\alp}{\alpha}
\newcommand{\sig}{\sigma}

\newcommand{\longc}{,\dotsc,}
\newcommand{\longp}{+\dotsb+}

\newcommand{\seq}{\subseteq}
\newcommand{\stm}{\setminus}

\newtheorem{lemma}{Lemma}
\newtheorem{theorem}{Theorem}
\newtheorem{corollary}{Corollary}

\newcommand{\refc}[1]{\ref{c:#1}}
\newcommand{\refl}[1]{\ref{l:#1}}
\newcommand{\reft}[1]{\ref{t:#1}}
\newcommand{\refs}[1]{\ref{s:#1}}
\newcommand{\refb}[1]{\cite{b:#1}}

\title{Sum-full sets are not zero-sum-free}

\author{Vsevolod F. Lev}
\email{seva@math.haifa.ac.il}
\address{Department of Mathematics, The University of Haifa at Oranim,
  Tivon 36006, Israel}

\author{J\'anos Nagy}
\thanks{JN was supported by the Lend\"ulet program of the Hungarian Academy
  of Sciences (MTA) and the NKFIH Grant \'Elvonal (Grant Nr. KKP126683).}
\email{janomo4@gmail.com}
\address{Alfr\'ed R\'enyi Institute of Mathematics and MTA-BME Lend\"ulet
  Arithmetic Combinatorics Research Group, E\"otv\"os Lor\'and Research Network.}

\author{P\'eter P\'al Pach}
\thanks{PPP was supported by the Lend\"ulet program of the Hungarian Academy
  of Sciences (MTA) and the National Research, Development and Innovation
  Office NKFIH (Grant Nr. K124171, K129335 and BME NC TKP2020).
 }
\email{ppp@cs.bme.hu}
\address{MTA-BME Lend\"ulet Arithmetic Combinatorics Research Group,
  E\"otv\"os Lor\'and Research Network,
  Department of Computer Science and Information Theory, Budapest
  University of Technology and Economics, 1117 Budapest, Magyar tud\'osok
  k\"or\'utja 2, Hungary.}

\makeatletter
\@namedef{subjclassname@2020}{\textup{2020} Mathematics Subject Classification}
\makeatother


\begin{document}
\baselineskip=16pt

\begin{abstract}
Let $A$ be a finite, nonempty subset of an abelian group. We show that if
every element of $A$ is a sum of two other elements, then $A$ has a
nonempty zero-sum subset. That is, a (finite, nonempty) sum-full subset of
an abelian group is not zero-sum-free.
\end{abstract}

\maketitle

\section{Introduction: Sum-full vs Zero-sum-free}

A subset $A$ of an abelian group is called \emph{sum-full} if every element
of $A$ is a sum of two other elements, possibly equal to each other; that is,
in the standard notation, if $A\seq 2A$. The subset is \emph{zero-sum} if the
sum of its elements is equal to $0$; it is \emph{zero-sum-free} if it does
not contain itself a non-empty zero-sum subset.

In these terms, a problem posted in 2010 by Gjergji Zaimi at the MathOverflow
web site~\refb{z} reads:
\begin{center}\emph{
Can a finite, nonempty, sum-full set of real numbers be zero-sum-free?}
\end{center}
Similar or related problems appeared in~\cite{b:s,b:v,b:ud}.

The problem has attracted some attention from the mathematical community,
perhaps due to the fact that \emph{a priori}, there is no obvious link
between the properties of being sum-full and being zero-sum-free. Despite the
interest, the only publication spanned by this problem we are aware of
is~\refb{br}, by Taras Banakh and Alex Ravsky.

It is immediately seen that a finite, nonempty, sum-full, zero-sum-free set
cannot contain zero, and that it must contain at least two negative and at
least two positive numbers (which can be made \emph{three} with a minor
effort). Furthermore, considering subsets of any torsion-free abelian group,
not necessarily the additive group of real numbers, leads to an equivalent
problem. Indeed, there is no reason to confine ourselves to torsion-free
groups; the same question can be asked for any abelian group. Apart from
several basic observations of this sort, the problem remained wide open: no
counterexample in any group has ever been found, and virtually no progress
towards a nonexistence result has been made.

In this note we present a complete solution to this problem.
\begin{theorem}\label{t:main}
Let $A$ be a finite, nonempty subset of an abelian group. If $A$ is sum-full,
then it is not zero-sum-free; that is, if every element of $A$ is
representable as a sum of two other elements, then $A$ has a nonempty
zero-sum subset.
\end{theorem}
The proof of Theorem~\reft{main} is of combinatorial nature; it is both
surprisingly short and elementary, requiring nothing beyond very basic linear
algebra.

Let $\cM_n$ denote the set of all integer square matrices of order $n$ with
all elements on the main diagonal greater than or equal to $-1$, all elements
off the main diagonal non-negative, and all row sums equal to $1$. The
following lemma is at the heart of the proof of Theorem~\reft{main}.
\begin{lemma}\label{l:main}
For any matrix $M\in\cM_n$ there exist nonzero vectors $u,v\in\{0,1\}^n$ such
that $M^tu=v$; that is, there exists a system of rows of $M$ such that their
sum is a nonzero, zero-one vector.
\end{lemma}

Deriving Theorem~\reft{main} from Lemma~\refl{main} is a matter of several
lines. Namely, given a finite, nonempty, sum-full subset $A=\{a_1\longc
a_n\}$ of an abelian group, for each $k\in[n]$ we fix a representation
$a_k=a_i+a_j$ with $i,j\in[n]$. Interpreting the resulting representations as
$n$-dimensional vectors, consider the matrix $M$ having these vectors as its
rows. (Thus, for instance, if in the representation $a_k=a_i+a_j$ the indices
$i,j,k$ are pairwise distinct, then $M$ contains three nonzero elements in
row $k$: the element on the main diagonal is equal to $-1$, and the elements
in columns $i$ and $j$ are equal to $1$.) Clearly, we have $M\in\cM_n$. By
the lemma, there is a nonzero vector $v=(v_1\longc v_n)\in\{0,1\}^n$
representable as a linear combination of the rows of $M$. On the other hand,
by the construction, every row of $M$ is orthogonal to the vector
$a:=(a_1\longc a_n)$; hence, $v$ is orthogonal to $a$, too. Denoting by $S$
the set of those indices $k\in[n]$ with $v_k=1$, we then get $\sum_{k\in S}
a_k=0$, showing that $A$ is not zero-sum-free.

We prove Lemma~\refl{main} in the next section. The proof is inductive; this
is why the class $\cM_n$ is slightly wider than the class of matrices which
can actually arise from sum-full sets.

Although Theorem~\reft{main} solves the original problem completely, in
Section~\refs{char3} we present yet another solution for the special case of
the elementary abelian $3$-groups. The reason for presenting a dedicated
solution in this particular case is that the solution is based on a totally
different approach relating sum-full sets and Sidon sets. Moreover, the
central lemma used in this case is in fact valid for any abelian group.
Postponing the details to Section~\refs{char3}, here we confine ourselves to
the remark that for the elementary abelian $2$-groups, the assertion of
Theorem~\reft{main} is almost immediate.

Finally, we note that Theorem~\reft{main} stays true in the more general
settings of \emph{sequences} (multisets) instead of sets. This follows by
observing that repeated elements can be removed from a set without acquiring
or violating the sumfullness property.

\section{Proof of Lemma~\refl{main}}\label{s:mainproof}

Our goal in this section is to prove Lemma~\refl{main}, and hence
Theorem~\reft{main}.

We use induction on $n$. For $n\le 2$ the assertion is easy to verify; assume
thus that $n\ge 3$. Write $M=(m_{ij})_{1\le i,j\le n}$ and for $i\in[n]$,
denote by $r_i$ the $i$th row of $M$.

If there exists $i\in[n]$ with $m_{ii}\ge 0$, then $r_i$ is a nonzero
zero-one vector; considering the row sum including $r_i$ as its only summand,
we get the assertion. We therefore assume that $m_{ii}=-1$ for all $i\in[n]$.

By the assumption just made, the sum of all off-diagonal elements of $M$ is
$2n$. If the sum in every column is $2$, then the sum of all the rows of $M$
is the all-one vector, completing the proof in this case. Assume that there
is a column with the sum of its off-diagonal elements $0$ or $1$. Changing
the order of the columns if necessary, we further assume that the first
column has this property.

If the sum of all off-diagonal elements of the first column is $0$, then the
first column contains only zero elements except that $m_{11}=-1$, and we
apply the induction hypothesis to the matrix obtained from $M$ by removing
its first row and first column. This yields a system of rows of this new
matrix with a nonzero zero-one sum. Prefixing $0$ at the beginning of each of
these rows, we get a system of rows of the \emph{original} matrix $M$ with a
nonzero, zero-one sum, as wanted.

Hence, we can assume that the sum of the off-diagonal elements of the first
column of $M$ is $1$; that is, one of the off-diagonal elements is $1$, while
the rest vanish. Without loss of generality, $m_{21}=1$ while $m_{j1}=0$ for
$j\in[3,n]$. In addition to the element $m_{21}=1$, the second row of $M$
contains an element equal to $1$ in a column other than the first or the
second one, and we assume, for definiteness, that it is the third column:
$m_{23}=1$. With all the assumptions made so far, $M$ can be visualized as
follows:
\newcommand{\s}{\ast}
\begin{equation*}
  M = \left(
  \begin{array}{rr|rrrr}
     -1 & \s\ & \s & \s & \dotsb & \s \\
      1 & -1\ &  1 &  0 & \dotsb & 0 \\
      \hline
      0 & \s\ & -1 & \s & \dotsb & \s \\
      0 & \s\ & \s & -1 & \dotsb & \s \\
      \vdots & \vdots\ & \vdots & \vdots & \ddots & \vdots \\
      0 & \s\ & \s & \s & \dotsb & -1
  \end{array}
  \right)
\end{equation*}

We now come to the crucial stage of the argument. Consider the square matrix
$M'$ of order $n-2$ obtained from $M$ by adding its second column to the
third one, and then removing the first two columns and first two rows of the
resulting matrix. Numbering the rows and columns of $M'$ from $3$ to $n$, we
write $M'=(m_{ij}')_{3\le i,j\le n}$.

Recalling that by $r_i$ we have denoted the $i$th row of $M$, let $r_i'$
denote the $i$th row of $M'$. It is easily seen that $M'\in\cM_{n-2}$; hence,
by the induction hypothesis, there is a subset $I\seq[3,n]$ such that the
vector $\sum_{i\in I}r_i'$ is nonzero and zero-one.

We remark that $m_{ii}'=m_{ii}=-1$ for all $i\in[4,n]$, but $m_{33}'\ge 0$ is
possible (this happens if and only if $m_{32}>0$). That is, all elements on
the main diagonal of $M'$, with the possible exception of $m_{33}'$, are
equal to $-1$.

For $j\in\{2,3\}$ let $\sig_j:=\sum_{i\in I}m_{ij}$. If $3\notin I$, then
both $\sig_2$ and $\sig_3$ are sums of non-negative terms; hence, are
nonnegative themselves. Also, $\sig_2+\sig_3=\sum_{i\in I}m'_{i3}\in\{0,1\}$.
This shows that $\sig_2,\sig_3\in\{0,1\}$, and, as a result,
 $\sum_{i\in I}r_i$ is a (nonzero) zero-one vector, as wanted.

We thus assume that $3\in I$. In this case we have $\sig_2+\sig_3\in\{0,1\}$
and $\sig_2\ge 0$. If also $\sig_3\ge 0$, then $\sig_2,\sig_3\in\{0,1\}$ and
we reach again the conclusion that $\sum_{i\in I}r_i$ is nonzero and
zero-one.

We are left with the situation where $3\in I$ and $\sig_3<0$. In this case we
have $\sig_2\in\{1,2\}$ and $\sig_3=-1$, and it follows that
 $r_2+\sum_{i\in I} r_i$ is a zero-one vector.

\section{The elementary abelian $3$-groups}\label{s:char3}

In this section we prove a version of Theorem~\reft{main} restricted to the
elementary abelian $3$-groups.
\begin{theorem}\label{t:char3}
Let $A$ be a finite, nonempty subset of an elementary abelian $3$-group. If
$A$ is sum-full, then it is not zero-sum-free.
\end{theorem}

Here is our main lemma.
\begin{lemma}\label{l:bb2}
Let $A$ be a finite, generating, sum-full, zero-sum-free subset of an abelian
group. Then for any proper subgroup $H$, there are $a_1,a_2,a_3,a_4\in A\stm
H$ with $a_1+a_2=a_3+a_4$ and $\{a_1,a_2\}\ne\{a_3,a_4\}$.
\end{lemma}

\begin{proof}
Since $A$ is generating and $H$ is proper, there is an element $a_1\in A\stm
H$. Since $A$ is sum-full, there are $a_2,b_1\in A$ with $a_1=a_2+b_1$. In
view of $a_1\notin H$, at least one of $a_2$ and $b_1$ is not in $H$;
switching $a_2$ and $b_1$ if needed, we assume that $a_2\notin H$. Iterating,
we get an infinite chain of equalities
\begin{align*}
  a_1 & = a_2 + b_1, \\
  a_2 & = a_3 + b_2, \\
  \quad & \quad \vdots
\end{align*}
with $a_1,a_2,\dotsc\in A\stm H$ and $b_1,b_2,\dotsc\in A$. Since $A$ is
finite, there are indices $1\le i<j$ such that $a_i=a_j$, while $a_s\ne a_t$
whenever $i\le s<t\le j-1$. Taking the sum of the corresponding equalities,
after a cancellation we obtain
  $$ b_i \longp b_{j-1} = 0. $$

Since $A$ is zero-sum-free, the summands in the left-hand side are not
pairwise distinct; thus we have, say, $b_\sig=b_\tau$ with
 $i\le\sig<\tau\le j-1$. This leads to
  $$ a_\sig - a_{\sig+1} = b_\sig = b_\tau = a_\tau - a_{\tau+1} $$
and then
  $$ a_\sig + a_{\tau+1} = a_{\sig+1} + a_\tau $$
proving the assertion unless either $a_\sig=a_\tau$, or $a_\sig=a_{\sig+1}$
hold. However, in the former case we get a contradiction with the assumption
$a_s\ne a_t$ ($i\le s<t\le j-1$), while in the latter case we have
$b_\sig=a_\sig-a_{\sig+1}=0$, contradicting the assumption that $A$ is
zero-sum-free.
\end{proof}

A subset $S$ of an abelian group is called a \emph{Sidon set} if it does not
contain nontrivial additive quadruples; that is, if an equality
$s_1+s_2=s_3+s_4$ with $s_1,s_2,s_3,s_4\in S$ is possible only for
$\{s_1,s_2\}=\{s_3,s_4\}$. Clearly, any independent set in a vector space,
and in particular any basis of a vector space, is Sidon.

The following corollary of Lemma~\refl{bb2} shows that a sum-full,
zero-sum-free set cannot be ``too close'' to a Sidon set.
\begin{corollary}\label{c:complement}
Suppose that $A$ is a finite, generating, sum-full, zero-sum-free subset of
an abelian group. If $B\seq A$ is Sidon (in particular, if $B$ is minimal
generating), then $A\stm B$ is generating.
\end{corollary}

For the proof, just apply the lemma to the subgroup $H$ generated by the set
$A\stm B$: having $H$ proper would produce a direct contradiction with the
assertion of the lemma in view of $A\stm H\seq A\stm(A\stm B)=B$.

The notion of a zero-sum-free set extends literally onto \emph{sequences}
with repeated elements allowed. Namely, a sequence is zero-sum-free if it
does not contain nonempty zero-sum subsequences. The following classical
result of Olson~\refb{o} establishes an upper bound for the largest possible
size of a zero-sum-free sequence in a finite vector space.
\begin{lemma}[Olson]\label{l:o}
Let $A$ be a finite sequence of vectors of the finite vector space $V$ over
the $p$-element field, with a prime $p$. If $A$ is zero-sum-free, then
$|A|\le (p-1)\dim V$.
\end{lemma}
Indeed, Olson has shown that for the finite abelian $p$-group with the
invariants $p^{\alp_1}\longc p^{\alp_m}$, the largest size of a zero-sum-free
sequence is $p^{\alp_1}\longp p^{\alp_m}-m$; this easily implies
Lemma~\refl{o}.

\begin{proof}[Proof of Theorem~\reft{char3}]
We consider the underlying group as a vector space over the field $\F_3$ and
denote it $V$.

Suppose that the theorem is wrong, and let $A\seq V$ be a counterexample;
that is, $A$ does not contain a zero-sum subset, while every element of $A$
is a sum of two other elements. Without loss of generality we assume that $A$
generates $V$; therefore, $V$ is finite.

Choose arbitrarily an element $a\in A$ and find a representation $a=b+c$ with
some (not necessarily distinct) $b,c\in A$.
It is easily verified that the set $\{a,b,c\}$ is Sidon; hence, by
Corollary~\refc{complement}, its complement $A\stm\{a,b,c\}$ is generating.
Fix a basis $B\seq A\stm\{a,b,c\}$. Applying Corollary~\refc{complement} a
second time, the complement $A\stm B$ is generating. Since some of the
vectors of $A\stm B$ (namely, $a,b$, and $c$) are linearly dependent, we have
  $$ \dim V < |A\stm B| = |A| - \dim V $$
whence $|A|>2\dim V$, contradicting Lemma~\refl{o}.
\end{proof}

\section*{Acknowledgement}

The first-named author is grateful to Mikhail Muzychuk for his interest and
fruitful conversations.


\bigskip
\end{document}